\documentclass[a4paper]{amsart}
\usepackage[latin1]{inputenc}
\usepackage{amssymb}
\usepackage{amsmath}
\usepackage{mathrsfs}
\usepackage{eufrak}
\usepackage{amsthm}
\usepackage{amsfonts}
\usepackage{textcomp}
\usepackage{graphicx}
\usepackage[pdftex]{color}
\usepackage{paralist}
\usepackage[shortlabels]{enumitem}
\usepackage{hyperref}
\usepackage{comment}
\usepackage[arrow, matrix, curve]{xy}
\usepackage{tikz}
\usepackage{tikz-cd}

\newtheorem*{corollary*}{Corollary}
\newtheorem*{conjecture*}{Conjecture}
\newtheorem*{example*}{Example}
\newtheorem*{theorem*}{Theorem}
\newtheorem*{proposition*}{Proposition}

\newtheorem{theorem}{Theorem}[section]

\newtheorem{lemma}[theorem]{Lemma}
\newtheorem{proposition}[theorem]{Proposition}

\newtheorem*{claim*}{Claim}

\newtheorem*{question}{Question}

\theoremstyle{definition}
\newtheorem{definition}[theorem]{Definition}

\theoremstyle{remark}

\numberwithin{equation}{section}

\makeatletter
\renewcommand*\env@matrix[1][\
arraystretch]{%
  \edef\arraystretch{#1}%
  \hskip -\arraycolsep
  \let\@ifnextchar\new@ifnextchar
  \array{*\c@MaxMatrixCols c}}
\makeatother

\newcommand{\Ext}{\operatorname{Ext}}

\newcommand{\gldim}{\operatorname{gldim}}
\newcommand{\End}{\operatorname{End}}

\newcommand{\pdim}{\operatorname{pdim}}
\newcommand{\idim}{\operatorname{idim}}

\newcommand{\gr}{\operatorname{grade}}
\newcommand{\cogr}{\operatorname{cograde}}

\newcommand{\Hom}{\operatorname{Hom}}

\renewcommand{\top}{\operatorname{\mathrm{top}}}
\newcommand{\rad}{\operatorname{\mathrm{rad}}}
\newcommand{\flatdim}{\operatorname{\mathrm{flatdim}}}
\newcommand{\soc}{\operatorname{\mathrm{soc}}}

\newcommand{\op}{\operatorname{op}}
\newcommand{\upset}[1]{\mathord{\uparrow}\!#1}
\newcommand{\downset}[1]{\mathord{\downarrow}\!#1}
\begin{document}

\title{A new characterisation of Auslander-Gorenstein algebras}
\date{\today}
\author[V.~Kl\'asz]{Vikt\'oria Kl\'asz}%
\address[V.~Kl\'asz]{Mathematical Institute of the University of Bonn, Endenicher Allee 60, 53115 Bonn, Germany}%
\email{klasz@math.uni-bonn.de}%

\author[M.~Kleinau]{Markus Kleinau}%
\address[M.~Kleinau]{Mathematical Institute of the University of Bonn, Endenicher Allee 60, 53115 Bonn, Germany}%
\email{mkleinau@math.uni-bonn.de}%

\author[R.~Marczinzik]{Ren\'e Marczinzik}%
\address[R.~Marczinzik]{Mathematical Institute of the University of Bonn, Endenicher Allee 60, 53115 Bonn, Germany}
\email{marczire@math.uni-bonn.de}

\author[J.~Marquardt]{Judith Marquardt}%
\address[J.~Marquardt]{Univ. Grenoble Alpes, CNRS, IF, 38000 Grenoble, France, and Universit\'e Paris-Saclay, UVSQ, CNRS, Laboratoire de Math\'ematiques de Versailles, 78000, Versailles, France}
\email{judith.marquardt@univ-grenoble-alpes.fr}



\begin{abstract}
We give a new characterisation of Auslander-Gorenstein finite dimensional algebras by showing that they are exactly the finite dimensional algebras with a well-defined Auslander-Reiten bijection. This proves a conjecture of Marczinzik. 
We give three applications of this new characterisation of Auslander-Gorenstein algebras. First we show that for certain classes of acyclic quiver algebras, the Auslander-Gorenstein property can be detected via the Bruhat factorisation of the Coxeter matrix. Second, we give a new proof that a finite lattice with an Auslander-Gorenstein incidence algebra has to be distributive. The last application answers a question of Iyama by showing that the diagonal Auslander regular property is not left-right symmetric.
\end{abstract}

\maketitle

\section*{Introduction}
A two-sided noetherian ring $A$ is said to be \emph{$n$-Gorenstein} for some $n \geq 1$ if there exists an injective coresolution $0 \rightarrow A \rightarrow I^0 \rightarrow I^1 \rightarrow \cdots $ such that $\flatdim I^i \leq i$ for all $0 \leq i < n$ and $A$ is said to satisfy the \emph{Auslander condition} if $A$ is $n$-Gorenstein for all $n \geq 1$. We remark that when $A$ is a finite dimensional $K$-algebra, the flat dimension simply coincides with the projective dimension.
$A$ is \emph{Auslander-Gorenstein} if it satisfies the Auslander condition and is \emph{Iwanaga-Gorenstein} in the sense that $\idim A_A= \idim {}_{A}A < \infty$. 
Auslander-Gorenstein rings are a non-commutative generalisation of the classical Gorenstein rings from commutative algebra due to Auslander, see \cite{B,FGR}.
Examples of Auslander-Gorenstein rings include enveloping algebras of finite dimensional Lie algebras, Weyl algebras, and rings of $\mathbb{C}$-linear differential operators on an irreducible smooth subvariety of an affine
space, see \cite[Chapter 3]{VO} for more details.
Auslander-Gorenstein algebras are also important in $p$-adic representation theory and related fields; see \cite{V,ST,AW}.
For the importance and applications of Auslander-Gorenstein rings, we refer to the two surveys \cite{Cl} and \cite{KM}. Especially for finite dimensional algebras the notion of Auslander-Gorenstein algebras played an increasingly important role and many new classes of Auslander-Gorenstein algebras were discovered in recent years, we mention for example blocks of category $\mathcal{O}$ \cite{KMM}, incidence algebras of distributive lattices \cite{IM}, higher Auslander algebras and generalisations \cite{I1,CIM,BHMPT} and partial classification results for monomial algebras \cite{K}.
The goal of this article is to give two new characterisations for Auslander-Gorenstein finite dimensional algebras.
In the following, we assume that all algebras are finite dimensional $K$-algebras for a field $K$ and modules are finitely generated right modules unless otherwise stated. Before we can state our main result, we give two definitions. The first definition can already be found in \cite{K,KM,KMT}.

\begin{definition}
An algebra $A$ is said to have a \emph{well-defined Auslander-Reiten map} if every indecomposable injective $A$-module $I$ has a finite minimal projective resolution 
$$0 \rightarrow P_{d_I}(I) \rightarrow \cdots \rightarrow P_0(I) \rightarrow I \rightarrow 0$$
with $P_{d_I}(I)$ indecomposable projective and $d_I$ the projective dimension of $I$. In this case, one defines the map $\psi$ from the (isomorphism classes of) indecomposable injective $A$-modules to the (isomorphism classes of) indecomposable projective $A$-modules as $\psi(I)=P_{d_I}(I)$.
When $\psi$ is bijective, we say that $A$ has a \emph{well-defined Auslander-Reiten bijection}.
\end{definition}
We remark that when $B$ is a representation-finite algebra with $M$ the direct sum of all indecomposable $B$-modules, then the Auslander algebra $A=\End_B(M)$ of $B$ is Auslander-Gorenstein and the Auslander-Reiten bijection of $A$ corresponds to the classical Auslander-Reiten translate on the indecomposable $B$-modules, we refer to \cite{MTY} for details.
The following definition is new:

\begin{definition}
An Iwanaga-Gorenstein algebra $A$ is said to have a \emph{well-defined grade map} if for every simple $A$-module $S$ with $g_S:=\gr S,$ we have that the module $h(S):=\top D \Ext_A^{g_S}(S,A)$ is simple. If the map $S \mapsto h(S)$ is additionally a bijection from (the isomorphism classes of) simple $A$-modules to (the isomorphism classes of) simple $A$-modules, then we say that $A$ has a \emph{well-defined grade bijection}.  
\end{definition}
Here we remark that the grade of an $A$-module $M$ is defined as $\gr M:= \inf \{ i \geq 0 \mid \Ext_A^i(M,A) \neq 0\}$ and this is always finite for Iwanaga-Gorenstein algebras, see for example \cite[Lemma 2.2]{KMT}.
It was first shown by Auslander and Reiten in \cite{AR} that Auslander-Gorenstein algebras have a well-defined Auslander-Reiten bijection and in \cite{I2} by Iyama that Auslander-Gorenstein algebras have a well-defined grade bijection. In \cite{KMT} it was shown that up to a canonical identification, the Auslander-Reiten bijection and the grade bijection for Auslander-Gorenstein algebras coincide, that is, we have 
\begin{equation*}
 \psi(I(S))=P(h(S)),
\end{equation*} where $I(S)$ denotes the injective envelope and $P(S)$ the projective cover of a simple module $S$.

The main result of this article is the following theorem:
\begin{theorem}
Let $A$ be a finite dimensional algebra. Then the following are equivalent:
\begin{enumerate}
    \item $A$ is Auslander-Gorenstein.
    \item $A$ is Iwanaga-Gorenstein and has a well-defined grade bijection $h$.
    \item $A$ has a well-defined Auslander-Reiten bijection $\psi$.
\end{enumerate}
In this case, we have 
\begin{equation*}
 \psi(I(S))=P(h(S)).
\end{equation*}
\end{theorem}
The direction that (1) implies (2) was already proven by Iyama in \cite{I2} and the equality \begin{equation*}
 \psi(I(S))=P(h(S))
\end{equation*} was proven in \cite{KMT}. The remaining implications are new and the equivalence of (1) and (3) was conjectured first by Marczinzik around 2022 \cite[Conjecture 1.2.4]{KM} and first proven for general monomial algebras by Kl\'asz in \cite{K}. 
We remark that recently, the Auslander-Reiten bijection for Auslander-Gorenstein algebras of finite global dimension was generalised to a wider class of acyclic quiver algebras, we refer for example to \cite{D+,KMT,KKM}.
We give three applications of our main result.
The first application gives a linear algebra criterion for the Auslander-Gorenstein condition for a certain class of acyclic quiver algebras, we refer to section 3 for details.
\begin{theorem} \label{C=PUtheorem}
\label{thm::C=PU_iff_ARbij}
    Let $A$ be a naturally labelled algebra that satisfies property $\circledast$. Then the following are equivalent:
    \begin{enumerate}[\rm (i)]
\item $A$ is Auslander regular.
        \item There exists a Bruhat decomposition of the Coxeter matrix  $C_A=U_1PU_2$ where $U_1=Id.$
    \end{enumerate}
In this case, the Coxeter permutation coincides with the Auslander-Reiten permutation.
\end{theorem}
The proof combines our main result with the results obtained in \cite[Section 4]{KKM}. 

As another application of our main result, we can give a new and much faster proof of the following result from \cite{IM}:
\begin{theorem}
Let $L$ be a finite lattice with incidence algebra $A$. If $A$ is Auslander-Gorenstein, then $L$ is distributive.
\end{theorem}
In forthcoming work we will also prove, using our new characterisation of Auslander-Gorenstein algebras,
that a certain class of posets that contains distributive lattices has Auslander-Gorenstein incidence algebras, which combined with our previous theorem generalises the main result of \cite{IM}.

A very practical application of our main result is that it allows one to test the Auslander-Gorenstein property in a much faster way than before using computer software such as the GAP package \cite{QPA}. We used this to search in a large class of algebras for an answer to the following question of Iyama in \cite[Section 3.6.2]{I3}
\begin{question}
Let $A$ be a diagonal Auslander regular algebra. Is then $A^{op}$ also diagonal Auslander regular?
\end{question}
We refer to section 3 for precise definitions.
We found a counterexample to this question by a combination of a large search for quiver algebras with a given Cartan matrix using \cite{Sage} and \cite{QPA} and ChatGPT Pro to find quiver algebras with those given Cartan matrices. We refer to section 3 for more details and the precise negative answer to Iyama's question, which is an acyclic quiver algebra with 9 simple modules.

\section{Preliminaries}

Throughout, $K$ denotes a field and $A$ a finite dimensional $K$-algebra. Unless explicitly stated otherwise, all modules are finitely generated right $A$-modules. We refer for example to the standard textbooks \cite{ARS,ASS,Ben,Kr,SY} on representation theory and homological algebra of finite dimensional algebras for the necessary background. We write $D=\Hom_K(-,K)$ and $M^*=\Hom_A(M,A)$ for a right $A$-module $M$. For a simple module $S$, we denote its projective cover by $P(S)$ and its injective envelope by $I(S)$. Recall that $\top P(S)\cong S$ and $\soc I(S)\cong S$, and that these assignments give bijections between the isomorphism classes of simple modules and those of indecomposable projective and indecomposable injective modules, respectively.

For a module $M$, we fix a minimal projective resolution
\[
 \cdots\longrightarrow P_2(M)\longrightarrow P_1(M)
 \longrightarrow P_0(M)\longrightarrow M\longrightarrow 0
\]
and write $P(M)=P_0(M)$. The terms in a minimal injective coresolution of $M$ are denoted by $I^i(M)$, with $I(M)=I^0(M)$. The \emph{grade} of $M$ is
\[
 \gr_A M=\inf\{i\geq 0\mid \Ext_A^i(M,A)\neq 0\},
\]
where the infimum of the empty set is understood to be infinity. We omit the subscript $A$ when there is no danger of confusion.
The next three lemmas are well-known and for the convenience of the reader we will sketch a proof for each of them.

\begin{lemma}\label{lem:duality-evaluation}
Let $P$ be a finitely generated projective right $A$-module, let $X$ be a right $A$-module, and let $L$ be a left $A$-module. There are natural isomorphisms of $K$-vector spaces
\[
 X\otimes_A P^*\cong\Hom_A(P,X)
 \qquad\text{and}\qquad
 \Hom_A(X,DL)\cong D(X\otimes_A L).
\]
\end{lemma}

\begin{proof}
The first map sends $x\otimes f$ to the homomorphism $p\mapsto xf(p)$. It is an isomorphism for a finite free module and hence also for each of its direct summands. The second is the usual tensor--Hom adjunction followed by $K$-duality; explicitly, a homomorphism $\alpha\colon X\to DL$ is sent to the functional $x\otimes \ell\mapsto\alpha(x)(\ell)$.
\end{proof}

\begin{lemma}\label{lem:minimal-ext-simple}
Let $M$ be a module and let $L$ be simple. Then the differentials in the complex $\Hom_A(P_\bullet(M),L)$ are zero. Consequently,
\[
 \Ext_A^i(M,L)\cong\Hom_A(P_i(M),L)
\]
for every $i\geq 0$.
\end{lemma}

\begin{proof}
Minimality gives $\operatorname{im}(P_{i+1}(M)\to P_i(M))\subseteq\rad P_i(M)$. Every homomorphism from $P_i(M)$ to the simple module $L$ vanishes on $\rad P_i(M)$. Hence composition with each differential is zero, and the assertion follows by taking the cohomology of $\Hom_A(P_\bullet(M),L)$.
\end{proof}

\begin{lemma}\label{lem:top-ext-epimorphism}
Let $X$ be a module of finite projective dimension $d$. If $Y\to Z$ is an epimorphism, then the induced map
\[
 \Ext_A^d(X,Y)\longrightarrow\Ext_A^d(X,Z)
\]
is an epimorphism.
\end{lemma}

\begin{proof}
Let $K$ be the kernel of $Y\to Z$. The long exact Ext sequence contains
\[
 \Ext_A^d(X,Y)\longrightarrow\Ext_A^d(X,Z)
 \longrightarrow\Ext_A^{d+1}(X,K).
\]
The last term is zero because $\pdim_A X=d$.
\end{proof}

The following truncation of a dual projective resolution will be used in the proof of the main theorem.

\begin{lemma}\label{lem:dual-truncation}
Let $M$ be a module of finite grade $g$, let $C^i=P_i(M)^*$, and set $E=\Ext_A^g(M,A)$. Define the left $A$-module
\[
 B=\operatorname{coker}(C^{g-1}\longrightarrow C^g),
\]
where $B=C^0$ when $g=0$. Then $E$ is a submodule of $B$, and the exact sequence
\[
 0\longrightarrow C^0\longrightarrow\cdots\longrightarrow C^g
 \longrightarrow B\longrightarrow 0
\]
gives a projective resolution of $B$ of length at most $g$. Moreover, for every right $A$-module $X$ and every $1\leq d\leq g$, there is an isomorphism
\[
 \Ext_A^d(X,DB)\cong D\Ext_A^{g-d}(M,X),
\]
while $\Ext_A^d(X,DB)=0$ for $d>g$. If $X$ is projective, then
\[
 \Hom_A(X,DE)\cong D\Ext_A^g(M,X).
\]
\end{lemma}

\begin{proof}
The cohomology of the complex $C^\bullet=\Hom_A(P_\bullet(M),A)$ is $\Ext_A^\bullet(M,A)$. Since $g=\gr_A M$, this complex is exact in every degree smaller than $g$. Thus the displayed sequence ending in $B$ is exact. The differential $C^g\to C^{g+1}$ factors through $B$, and
\[
 \ker(C^g\to C^{g+1})/\operatorname{im}(C^{g-1}\to C^g)
 =\Ext_A^g(M,A)=E.
\]

This proves that $E$ is a submodule of $B$. By definition of $B$ and the condition that $M$ has grade equal to $g$, the exact sequence
\[
 0\longrightarrow C^0\longrightarrow\cdots\longrightarrow C^g
 \longrightarrow B\longrightarrow 0
\]
gives a projective resolution of $B$ of length at most $g$. Since every $C^i$ is projective as a left module, every $DC^i$ is injective as a right module, so dualising gives an injective resolution of $DB$.

For $1\leq d\leq g$, the term in degree $d$ of this injective resolution is $DC^{g-d}$. Lemma \ref{lem:duality-evaluation} gives, term by term,
\[
 \Hom_A(X,DC^{g-d})\cong D(X\otimes_A C^{g-d})
 \cong D\Hom_A(P_{g-d}(M),X),
\] using here that every projective module $P$ is reflexive, so that we have $P^{**} \cong P$.
These isomorphisms identify the complex obtained after applying $\Hom_A(X,-)$ with the $K$-dual of $\Hom_A(P_\bullet(M),X)$, read in the reverse order. Since $d\geq 1$, the truncation at $C^g$ does not change the cohomology in degree $g-d$. As $K$-duality is exact, cohomology in degree $d$ is therefore the dual of the cohomology in degree $g-d$. This gives the claimed isomorphism. The vanishing for $d>g$ follows from the length of the injective resolution.

Finally, assume that $X$ is projective. Then $X\otimes_A-$ is exact and therefore commutes with the cohomology of $C^\bullet$. Lemma \ref{lem:duality-evaluation} gives
\[
 X\otimes_A E\cong H^g(X\otimes_A C^\bullet)
 \cong H^g(\Hom_A(P_\bullet(M),X))
 =\Ext_A^g(M,X).
\]
Applying the second isomorphism in Lemma \ref{lem:duality-evaluation} completes the proof.
\end{proof}

\begin{theorem} \label{FGRandARresults}
Let $A$ be a finite dimensional algebra.
\begin{enumerate}
    \item If $A$ satisfies the Auslander condition, then so does $A^{op}$.
    \item If $A$ satisfies the Auslander condition and $\idim A_A < \infty$, then $A$ is Iwanaga-Gorenstein.
\end{enumerate}
\end{theorem}
\begin{proof}
For (1), see \cite[Theorem 3.7]{FGR} and for (2), see \cite[Corollary 5.5 (b)]{AR}.
\end{proof}

\section{Equivalent characterisations for Auslander-Gorenstein algebras}
In this section, we prove the main result of this article:
\begin{theorem} \label{thm:mainresult}
Let $A$ be a finite dimensional algebra. Then the following are equivalent:
\begin{enumerate}
    \item $A$ is Auslander-Gorenstein.
    \item $A$ is Iwanaga-Gorenstein and has a well-defined grade bijection $h$.
    \item $A$ has a well-defined Auslander-Reiten bijection $\psi$.
\end{enumerate}
In this case, we have 
\begin{equation*}
 \psi(I(S))=P(h(S)).
\end{equation*}
\end{theorem}

Note that the direction that (1) implies (2) was already shown by Iyama in \cite{I2} and so we will show that (2) implies (3) and (3) implies (1).

We split the proof in two subsections. 
\subsection{(2) implies (3)}

We first need two lemmas.

\begin{lemma}\label{lem:endpoint-evaluation}
Let $M$ be a module of finite grade $g$, set $E=\Ext_A^g(M,A)$, and let $X$ be an injective module of finite projective dimension $d$. If $\Ext_A^d(X,DE)\neq 0$, then $d=g$ and $\Hom_A(M,X)\neq 0$. Whenever $d=g$, there is a $K$-linear epimorphism
\[
 D\Hom_A(M,X)\longrightarrow\Ext_A^d(X,DE).
\]
\end{lemma}

\begin{proof}
As in Lemma \ref{lem:dual-truncation}, set $C^i=P_i(M)^{*}$ and \[
 B=\operatorname{coker}(C^{g-1}\longrightarrow C^g),
\]
where $B=C^0$ when $g=0$. The inclusion $E\subseteq B$ dualises to an epimorphism $DB\to DE$.

Assume first that $d>0$. Lemma \ref{lem:top-ext-epimorphism} gives an epimorphism
$\Ext_A^d(X,DB)\to\Ext_A^d(X,DE)$. Hence the hypothesis implies that $\Ext_A^d(X,DB)$ is nonzero. By Lemma \ref{lem:dual-truncation}, this is impossible for $d>g$, while for $d\leq g$ it is equivalent to the non-vanishing of $\Ext_A^{g-d}(M,X)$. Since $X$ is injective, the latter can occur only when $g-d=0$. Thus $d=g$ and $\Hom_A(M,X)\neq 0$. If $d=g>0$, Lemma \ref{lem:dual-truncation} identifies $D\Hom_A(M,X)$ with $\Ext_A^d(X,DB)$, and Lemma \ref{lem:top-ext-epimorphism} shows that the map $\Ext_A^d(X,DB)\to\Ext_A^d(X,DE)$ induced by $DB\to DE$ is surjective. Their composite is the required epimorphism.

If $d=0$, then $X$ is projective-injective. The last assertion of Lemma \ref{lem:dual-truncation} gives
$\Hom_A(X,DE)\cong D\Ext_A^g(M,X)$. Its left-hand side is nonzero by assumption, whereas injectivity of $X$ makes the right-hand side zero for $g>0$. Hence $g=0=d$, and the same isomorphism becomes $\Hom_A(X,DE)\cong D\Hom_A(M,X)$.
\end{proof}

\begin{lemma}\label{lem:corner}
Let $R$ and $S$ be simple modules. Let $I=I(S)$, $d=\pdim_A I$, $g=\gr_A R$, $E=\Ext_A^g(R,A)$ and $U=\top DE$. Assume that $d$ and $g$ are finite, that $U$ is simple, and that $P(U)$ is a direct summand of $P_d(I)$. Then $d=g$ and $R\cong S$.
\end{lemma}

\begin{proof}
By Lemma \ref{lem:minimal-ext-simple},
\[
 \Ext_A^d(I,U)\cong\Hom_A(P_d(I),U)\neq 0;
\]
the last inequality follows by restricting to the summand $P(U)$ and using the projective cover $P(U)\to U$. The canonical epimorphism $DE\to U$ induces an epimorphism
$\Ext_A^d(I,DE)\to\Ext_A^d(I,U)$ by Lemma \ref{lem:top-ext-epimorphism}. Hence $\Ext_A^d(I,DE)\neq 0$, and Lemma \ref{lem:endpoint-evaluation}, applied to $M=R$ and $X=I$, yields $d=g$ and $\Hom_A(R,I(S))\neq 0$. A nonzero map from the simple module $R$ to $I(S)$ identifies $R$ with a simple submodule of $I(S)$. Since $\soc I(S)\cong S$, we obtain $R\cong S$.
\end{proof}

We can now prove the main result of this subsection.
\begin{theorem}
Let $A$ be an Iwanaga--Gorenstein algebra with a well-defined, bijective grade map. Then $A$ has a well-defined, bijective Auslander--Reiten map.
\end{theorem}

\begin{proof}
Let $S$ be simple, set $I=I(S)$ and $d=\pdim_A I$, which is finite since $A$ is assumed to be Iwanaga-Gorenstein. Let $P(V)$ be an indecomposable direct summand of $P_d(I)$, where $V$ is simple. Surjectivity of the grade map gives a simple module $R$ with $h(R)\cong V$. Lemma \ref{lem:corner} then gives $R\cong S$ and $d=g_S$. Thus every indecomposable direct summand of $P_d(I)$ is isomorphic to $P(h(S))$. By the Krull--Schmidt Theorem,
\[
 P_d(I(S))\cong P(h(S))^{m_S}
\]
for some positive integer $m_S$.

It remains to determine this multiplicity $m_S$. Set $U=h(S)$ and $E=\Ext_A^d(S,A)$. Since $d=g_S$, Lemma \ref{lem:endpoint-evaluation} and the canonical epimorphism $DE\to U$, together with Lemma \ref{lem:top-ext-epimorphism}, give a composite epimorphism
\[
 D\Hom_A(S,I(S))\longrightarrow\Ext_A^d(I(S),DE)
 \longrightarrow\Ext_A^d(I(S),U).
\]
By Lemma \ref{lem:minimal-ext-simple} and the decomposition of $P_d(I(S))$, the last term is isomorphic to $\Hom_A(P(U),U)^{m_S}$. Moreover, $\Hom_A(S,I(S))\cong\End_A(S)$ because $S$ is the socle of $I(S)$, while $\Hom_A(P(U),U)\cong\End_A(U)$ because $U$ is the top of $P(U)$. Comparing dimensions therefore gives
\[
 m_S\cdot\dim_K\End_A(h(S))\leq\dim_K\End_A(S).
\]

Let $S_0,S_1,\ldots,S_{\ell-1}$ be a cycle of the permutation $h$, with $S_{i+1}=h(S_i)$ and indices read modulo $\ell$. Multiplying the preceding inequalities along this cycle cancels the dimensions of the endomorphism rings and gives $\prod_i m_{S_i}\leq 1$. Since each $m_{S_i}$ is a positive integer, all of them are equal to one. Consequently,
\[
 P_{g_S}(I(S))\cong P(h(S))
\]
for every simple module $S$.

It follows that $\psi(I(S))=P(h(S))$. The assignments $S\mapsto I(S)$ and $U\mapsto P(U)$ are bijections on isomorphism classes, and $h$ is a permutation of the simple modules. Hence the Auslander--Reiten map is well defined and bijective.
\end{proof}

\subsection{(3) implies (1)}
\begin{lemma} \label{mainlemma}
Let $A$ be an algebra with a well-defined bijective Auslander-Reiten map $\psi$. 
Let $P$ be an indecomposable projective $A$-module with $\pdim \psi^{-1}(P)=d>0$ and let $N$ be a right $A$-module with $\idim N < d$. Then every $A$-module homomorphism $f: P \rightarrow N$ satisfies $f(P) \subseteq \rad N$.
\end{lemma}
\begin{proof}
Set $I:=\psi^{-1}(P)$, which is an indecomposable injective $A$-module by our assumption that $A$ has a well-defined bijective Auslander-Reiten map.
Take a minimal projective resolution of $I$ as follows:
$$0 \rightarrow P \xrightarrow{u} P_{d-1} \rightarrow P_{d-2} \rightarrow \cdots \rightarrow P_0 \rightarrow I \rightarrow 0.$$
Note that $\pdim I=d< \infty$ is immediate from the well-definedness of $\psi$.
Since the resolution is minimal, we have $u(P) \subseteq \rad P_{d-1}$. Define $C$ to be the cokernel of the map $u$. Then there is a short exact sequence 
$$0 \rightarrow P \xrightarrow{u} P_{d-1} \rightarrow C \rightarrow 0.$$
Now, by definition we have $C \cong \Omega^{d-1}(I)$ and thus by dimension shifting
$$\Ext_A^1(C,N) \cong \Ext_A^d(I,N)=0,$$
since by assumption $N$ has injective dimension strictly smaller than $d$.
Now apply the functor $\Hom_A(-,N)$ to the short exact sequence 
$$0 \rightarrow P \xrightarrow{u} P_{d-1} \rightarrow C \rightarrow 0.$$
to obtain the exact sequence
$$\Hom_A(P_{d-1},N) \rightarrow \Hom_A(P,N) \rightarrow \Ext_A^1(C,N).$$
As $\Ext_A^1(C,N)=0$, the restriction map $\Hom_A(P_{d-1},N) \rightarrow \Hom_A(P,N)$ is surjective. In other words, for every map $f: P \rightarrow N$ there exists a map $h: P_{d-1} \rightarrow N$ such that $f=h \circ u$.
This gives us that $f(P)=h(u(P)) \subseteq h(\rad P_{d-1}) \subseteq \rad N$, the desired inclusion.
\end{proof}

\begin{theorem}
Let $A$ be an algebra with a well-defined bijective Auslander-Reiten map.
Then for every integer $r \geq 0$:
\begin{enumerate}
    \item For every indecomposable injective $A$-module $I$, we have $\idim \Omega^r(I) \leq r$.
    \item If $P$ is an indecomposable direct summand of $P_r(I)$ in the minimal projective resolution of an indecomposable injective $A$-module $I$, then $\pdim \psi^{-1}(P) \leq r$ and $\idim P \leq r$.
\end{enumerate}
In particular, $A$ is Auslander-Gorenstein.
\end{theorem}
\begin{proof}
We prove (1) and (2) simultaneously by induction on r. \newline
\underline{Base Case $r=0$:} Then $\Omega^0(I)=I$ and $\idim I=0$, so (1) is clear for $r=0$. For (2), let $P$ be an indecomposable direct summand of the projective cover $P_0(I)$ of $I$. 
By the definition of projective covers, the image of the component map $P \rightarrow I$ is not contained in $\rad I$.
Assume $\pdim \psi^{-1}(P)>0$, which implies $\idim I =0 < \pdim \psi^{-1}(P)$ and we can apply Lemma \ref{mainlemma} with $N=I$, which tells us that every map $P \rightarrow I$ has image contained in $\rad I$, a contradiction. Thus $\pdim \psi^{-1}(P)=0$. This immediately gives us that $P \cong \psi^{-1}(P)$ and thus $P$ is projective-injective and thus $\idim P=0$. \newline
\underline{Induction step:}
Assume assertions (1) and (2) have already been proved in every degree strictly smaller than $r$. \newline
\underline{Proof for degree $r$:}
First we prove that (1) is satisfied: Let $I$ be an indecomposable injective $A$-module and consider
the exact sequence 
\begin{equation}\label{syzygies}
0 \longrightarrow \Omega^r(I) \longrightarrow P_{r-1}(I)
\longrightarrow \Omega^{r-1}(I) \longrightarrow 0.
\end{equation}
By induction hypothesis, every indecomposable direct summand of $P_{r-1}(I)$ has injective dimension at most $r-1$ and thus $\idim P_{r-1}(I) \leq r-1$ and we also have $\idim \Omega^{r-1}(I) \leq r-1$.
Let $J = \rad(A)$.
For a general $A$-module $X$, we have $\idim X= \sup \{ t \geq 0 \mid \Ext_A^t(A/J,X) \neq 0 \}$.
Applying the functor $\Hom_A(A/J,-)$ to the short exact sequence \ref{syzygies} and looking at the long exact Ext-sequence we get for $t \geq r$
$$\cdots \rightarrow \Ext_A^{t}(A/J,\Omega^{r}(I)) \rightarrow \Ext_A^{t}(A/J,P_{r-1}(I)) \rightarrow \Ext_A^{t}(A/J,\Omega^{r-1}(I)) \rightarrow  $$$$\Ext_A^{t+1}(A/J,\Omega^{r}(I)) \rightarrow 
\Ext_A^{t+1}(A/J,P_{r-1}(I)) \rightarrow \cdots .$$
Using $\Ext_A^{t}(A/J,P_{r-1}(I))=0$ and $\Ext_A^{t}(A/J,\Omega^{r-1}(I))=0$ for $t \geq r$, this shows that for $t \geq r$  we have $\Ext_A^{t+1}(A/J,\Omega^{r}(I))=0$, and thus, $\idim \Omega^r(I) \leq r$ and (1) is proven in degree $r$.
Now let $P$ be an indecomposable direct summand of $P_r(I)$ and consider the projective cover $P_r(I) \rightarrow \Omega^r(I)$ of $\Omega^r(I)$. 
We show that $\pdim \psi^{-1}(P) \leq r$. Assume otherwise that $\pdim \psi^{-1}(P) >r$, which gives by the already proven case (1) that $\idim \Omega^r(I) \leq r < \pdim \psi^{-1}(P)$. Now we apply Lemma \ref{mainlemma} with $N=\Omega^r(I)$ to conclude that $P \rightarrow \Omega^r(I)$ has image contained in $\rad \Omega^r(I)$, which contradicts the properties of the projective cover. Thus we must have $\pdim \psi^{-1}(P) \leq r$.
In the final step we must prove that $\idim P \leq r$. We consider here two cases. \newline
\textbf{Case 1:} First assume that $\pdim \psi^{-1}(P)<r$. Let $R:= \psi^{-1}(P)$, which is indecomposable injective by assumption that $A$ has a well-defined bijective Auslander-Reiten map. Set $d:=\pdim \psi^{-1}(P)$.
Then $P$ occurs as the final non-zero indecomposable term in degree $d$ of the minimal projective resolution of $R$ and thus $P$ is an indecomposable direct summand of $P_d(R)$ (in fact, it is isomorphic to that). Since we are in the case $\pdim \psi^{-1}(P) <r$, the integer $d$ is strictly smaller than the current induction degree $r$ and, by induction hypothesis, the assertion (2) applies. \newline
\textbf{Case 2:} Assume now that $\pdim \psi^{-1}(P)=r$.
Set $R:=\psi^{-1}(P)$ and consider 
$$0 \rightarrow P \rightarrow P_{r-1}(R) \rightarrow \Omega^{r-1}(R) \rightarrow 0.$$
By induction hypothesis and (2), every indecomposable summand of $P_{r-1}(R)$ has injective dimension at most $r-1$ and thus $\idim P_{r-1}(R) \leq r-1$ and furthermore by induction hypothesis and (1) we know that $\idim \Omega^{r-1}(R) \leq r-1$. As before, a long Ext-sequence argument shows then that $\idim P \leq r$ and thus (2) is also proven by induction.

Now let $A$ be an algebra with a well-defined, bijective Auslander-Reiten map and let 
$$0 \rightarrow P_n(D(A)) \rightarrow \cdots \rightarrow P_0(D(A)) \rightarrow D(A) \rightarrow 0$$
be a minimal projective resolution of $D(A)$, which is obtained as a direct sum of minimal projective resolutions of the indecomposable injective $A$-modules.
Since $\psi$ is well-defined and bijective, this resolution of $D(A)$ is finite.
By (2), $\idim P \leq r$ for every indecomposable direct summand $P$ of $P_r(D(A))$ and thus $\idim P_r(D(A)) \leq r$ for all $r \geq 0$. Thus, $A^{op}$ satisfies the Auslander condition.
By Theorem \ref{FGRandARresults},  with $A^{op}$, also $A$ satisfies the Auslander condition and is in fact Auslander-Gorenstein since every indecomposable injective $A$-module has finite projective dimension, which is equivalent to every indecomposable projective $A^{op}$-module having finite injective dimension by duality.
\end{proof}

\section{Three applications}

\subsection{A Coxeter matrix criterion for being Auslander regular}
In this section let $A=KQ/I$ be finite dimensional quiver algebra with admissible ideal $I$ with $n$ simple modules and primitive idempotents $e_i$ corresponding to the vertices of $Q$. Note that over algebraically closed fields $K$, every finite dimensional $K$-algebra is Morita equivalent to such a quiver algebra and thus our assumption that $A$ is a quiver algebra is no loss of generality in this case. Recall that an algebra $A$ is called \emph{Auslander regular} if $A$ is Auslander-Gorenstein of finite global dimension. The \emph{Cartan matrix} $W$ of $A$ is defined as the $n \times n$-matrix with entries $W_{i,j}=\dim e_j A e_i$ and the \emph{Coxeter matrix} of $A$ is defined as $C:=-W^T W^{-1}$. 
Now assume in this subsection that $Q$ is an acyclic quiver.
A labelling by $\{1,...,n\}$ of the vertices of $A$ is called a \emph{natural labelling} if $e_i A e_j \neq 0$ implies that $i \leq j$. Note that this is equivalent to the Cartan matrix of $A$ being lower triangular.
The \emph{Coxeter permutation} of a naturally labelled acyclic quiver algebra $A$ is defined as the permutation corresponding to the permutation matrix $P$ such that $C=U_1 P U_2$ is a Bruhat factorisation of the Coxeter matrix $C$ of $A$ with upper triangular matrices $U_1$ and $U_2$. Note that this depends on the chosen natural labelling in general, but it was shown in \cite[Theorem 4.6]{KKM} that for Auslander regular algebras $A$ the Coxeter permutation is independent of the chosen natural labelling and coincides with the Auslander-Reiten permutation, we refer to \cite[Section 4]{KKM} for details and examples. Note here that we identity the Auslander-Reiten bijection on an Auslander regular algebra with a permutation, called the Auslander-Reiten permutation, on the simple modules or equivalently on the vertices of the quiver algebra, which we can do by the formula 
\begin{equation*}
 \psi(I(S))=P(h(S)).
\end{equation*}
from theorem \ref{thm:mainresult}.
In \cite{KKM} it was proved that a naturally labelled acyclic Auslander regular algebra $A$ has the property that the Coxeter matrix can be written as $C=PU$ for a permutation matrix $P$ and an upper triangular matrix $U$. In other words: Up to reordering the rows of $C$, $C$ is upper triangular.
In \cite{KMT} it was shown that the condition that there exists a factorisation $C=PU$ is even equivalent to being Auslander regular for incidence algebras of finite lattices and in \cite{KKM} this was proved also for certain monomial algebras.
In general, it is an open problem under which precise conditions
Auslander regularity is equivalent to the existence of a factorisation
$C=PU$ for a naturally labelled acyclic algebra.
We give the following definition, following \cite[Definition 4.3]{KKM}:
\begin{definition}
Let $A=KQ/I$ be a naturally labelled algebra.
An $A$-module $M$ \emph{satisfies property $\circledast$} if the following property is satisfied:
Let $0 \rightarrow M \rightarrow I^0 \rightarrow \cdots \rightarrow I^n \rightarrow 0 $ be a minimal injective coresolution of $M$, then $I^n=I(x)$ is an indecomposable module and $x <y$ for all $y$ such that $I(y)$ is a direct summand of $I^i$ for some $0 \leq i <n$.
We say that an algebra $A$ \emph{satisfies property $\circledast$} if every indecomposable projective $A$-module satisfies property $\circledast$.
\end{definition}
Examples of naturally labelled acyclic algebras satisfying property $\circledast$ include Auslander regular algebras, 2-Gorenstein monomial algebras and linear Nakayama algebras, we refer to \cite[Section 4]{KKM}.
Combining our main result in theorem \ref{thm:mainresult} with the results of \cite[Section 4]{KKM}, we obtain the following result:
\begin{theorem} \label{C=PUtheorem}
\label{thm::C=PU_iff_ARbij}
    Let $A$ be a naturally labelled algebra that satisfies property $\circledast$. Then the following are equivalent:
    \begin{enumerate}
\item $A$ is Auslander regular.
        \item There exists a Bruhat decomposition of the Coxeter matrix  $C_A=U_1PU_2$ where $U_1=Id.$
    \end{enumerate}
In this case, the Coxeter permutation coincides with the Auslander-Reiten permutation.
\end{theorem}
\begin{proof}
It was proven in \cite[Theorem 4.5]{KKM} that condition (2) is equivalent to $A^{op}$ having a well-defined Auslander-Reiten bijection, which by our main result theorem \ref{thm:mainresult} is equivalent to $A^{op}$ being Auslander regular, which is equivalent to $A$ being Auslander regular by theorem \ref{FGRandARresults}.
\end{proof}
For example, it is elementary to see that linear Nakayama algebras (which have unique natural labelling, which we assume in the following) satisfy property $\circledast$ and thus a linear Nakayama algebra $A$ is Auslander regular if and only if we have a factorisation $C=PU$ for the Coxeter matrix of $A$ as was already proven in \cite{KKM}. This translates the homological Auslander condition to simple linear algebra in this case. 
It would be interesting to know whether one can relax property $\circledast$ in theorem \ref{C=PUtheorem} or replace it with other conditions that include other interesting classes of algebras.

\subsection{The Auslander-Gorenstein property for incidence algebras of finite lattices}
In this section we give a new proof of the fact that a finite lattice $L$ with Auslander-Gorenstein incidence algebra $A$ must be distributive. 
This was first proven in \cite[Theorem 5.4]{IM} by a long case distinction. Here we will give a much faster proof using the two facts that a finite lattice is distributive if and only if its join-irreducible elements coincide with its join-prime elements and the fact that being Auslander-Gorenstein is the same as having a well-defined Auslander-Reiten bijection by the main result of this article.
We use the conventions and notions of \cite{IM} for lattices and their incidence algebras. 
Recall that the simple modules of $A$ correspond to vertices of the lattice $L.$ 
For an $x\in L$, $S(x), P(x)$ and $I(x)$ denote the simple, indecomposable projective, and indecomposable injective $A$-module associated to $x,$ respectively. 
For a finite lattice $L$, the minimal element is denoted by $m$ and the maximal element by $M$. Additionally, we define for an element $x$ in a lattice $L$ the subset $\downset{x}  := \{y \in L \mid y \leq x \}$ and $\upset{x}  := \{y \in L \mid y \geq x \}.$
We will use the elementary fact that an element $j \in L$ is join-irreducible if and only if it covers a unique element in $L$. We will denote this unique element by $j_{*}$ in the following.
\begin{lemma} \label{mainlemmaARbijimpliesdist}
Let $j$ be a join-irreducible element of the finite lattice $L$ and assume that the last nonzero term in the minimal projective resolution of the $A$-module $I(x)$ is $P(j)$. Then $\pdim I(x)=1$ and $L \setminus \downset{x}=\upset{j}$.
\end{lemma}
\begin{proof}
Since $j$ is join-irreducible, we have $\downset j \setminus \{j \}= \downset j_*$ and this gives the minimal injective coresolution 
$$0 \rightarrow S(j) \rightarrow I(j) \rightarrow I(j_*) \rightarrow 0.$$
Thus $\idim S(j)=1$.
Now let $I(x)$ be an indecomposable injective $A$-module with nonzero last term $P(j)$ in its minimal projective resolution and $\pdim I(x)=d$.
Then using Lemma \ref{lem:minimal-ext-simple} $\Ext_A^d(I(x),S(j))=\Hom_A(P(j),S(j)) \neq 0$, which implies $d \leq 1$ as $\idim S(j)=1$.
As the last term in the minimal projective resolution of $I(x)$ is $P(j)$, which is not projective-injective, we must have $d=1$.
We have the short exact sequence 
$$0 \rightarrow P(j) \rightarrow P(m) \rightarrow I(x) \rightarrow 0,$$
which directly gives by comparing dimension vectors of those modules that $L \setminus \downset{x}=\upset{j}$.
\end{proof}
Recall that an element $j$ in a finite lattice $L$ is called \emph{join-prime} if it is not the minimal element and 
\[
 j\leq a\vee b\quad\Longrightarrow\quad
 j\leq a\ \text{or}\ j\leq b.
\] Note that a join-prime element is always join-irreducible, see for example \cite[Chapter 9.3]{Gar}.
The following lemma gives an important characterisation of distributive lattices among finite lattices:
\begin{lemma} \label{distributive lattice characterisationlemma}
Let $L$ be a finite lattice.
Then the following two conditions are equivalent:
\begin{enumerate}
    \item $L$ is distributive.
    \item For every $x \in L$, $x$ is join-irreducible if and only if $x$ is join-prime.
\end{enumerate}
\end{lemma}
\begin{proof}
See for example \cite[Theorem 5.1]{CLM}.
\end{proof}
We can now prove the main result of this section:
\begin{theorem}
Let $L$ be a finite lattice whose incidence algebra $A$ is Auslander-Gorenstein. Then $L$ is distributive.
\end{theorem}
\begin{proof}
We use that by our main result, Theorem \ref{thm:mainresult}, being Auslander-Gorenstein is equivalent to having a well-defined Auslander-Reiten bijection.
Let $j$ be join-irreducible. We show that $j$ is join-prime and then $L$ is distributive by Lemma \ref{distributive lattice characterisationlemma}. Since by assumption, the Auslander-Reiten map is surjective there is an indecomposable injective module $I(x)$ with last nonzero term $P(j)$ in its minimal projective resolution. By Lemma \ref{mainlemmaARbijimpliesdist} we have $L \setminus \downset{x}=\upset{j}$. 
Suppose \(j\leq a\vee b\), but \(j\nleq a\) and \(j\nleq b\).  Then
\(a,b\notin\upset{j}\). Now $L \setminus \downset{x}=\upset{j}$ implies
\[
 a,b\in\downset{x},
\]
so \(a\leq x\) and \(b\leq x\).  Hence \(a\vee b\leq x\), and thus
\[
 j\leq a\vee b\leq x.
\]
However, \(j\in\upset{j}=L\setminus\downset{x}\), so \(j\nleq x\).
This contradiction proves that \(j\) is join-prime.
\end{proof}
In forthcoming work \cite{IKKM} we will give a general classification of finite posets with 2-Gorenstein incidence algebras, which can be used to generalise the theorem in this subsection from lattices to general posets.

\subsection{On diagonal Auslander regular algebras}
In \cite[Section 3.4]{I3}, Iyama introduced the notion of diagonal Auslander regular algebras as follows:
\begin{definition}
Let $A$ be a finite dimensional algebra
with minimal injective coresolution
$$0 \rightarrow A_A \rightarrow I^0 \rightarrow I^1 \rightarrow \cdots .$$
Then $A$ is called \emph{diagonal Auslander regular} if $\gldim A < \infty$ and every indecomposable direct summand $X$ of $I^i$ has projective dimension $i$ for all $i \geq 0$.
\end{definition}
We remark that it was shown in \cite[Corollary 3.4]{KMT} that an Auslander regular algebra $A$ is diagonal Auslander regular if and only if every simple module $S$ is \emph{perfect} in the sense that it satisfies $\gr S= \pdim S.$
The class of diagonal Auslander regular algebras is rather mysterious: While Iyama classified them in the global dimension $\leq 2$ case in \cite{I3} and it was shown in \cite{IM} that incidence algebras of distributive lattices are diagonal Auslander regular, no other large class of diagonal Auslander regular algebras is currently known.
In \cite[Section 3.6.2]{I3}, Iyama posed the following question:
\begin{question}
If $A$ is a diagonal Auslander regular algebra, is the same true for $A^{op}$?
\end{question}
Since being Auslander regular is left right symmetric, this question can be equivalently stated as follows: If $A$ is an Auslander regular algebra with all simple $A$-modules being perfect, are then all simple $A$-modules coperfect? Here \emph{coperfect} for an $A$-module $M$ means that $\operatorname{cograde} M := \inf \{ i \geq 0 \mid \Ext_A^i(D(A),M) \neq 0 \}$ coincides with the injective dimension $\idim M$ of $M$.

The goal of this section is to give a negative answer to the question by Iyama. This was obtained using the GAP-package \cite{QPA}, \cite{Sage} and ChatGPT pro. We will briefly explain how we obtained this counterexample in the following. 
The starting point is the following proposition:
\begin{proposition} \label{diagausregpropo}
Let $A$ be a naturally labelled diagonal Auslander regular algebra with Cartan matrix $W$ and Coxeter matrix $C=-W^T W^{-1}$.
Then the following conditions hold:
\begin{enumerate}
    \item $C=PU$ for a permutation matrix $P$ and an upper triangular matrix $U$.
    \item The non-zero entries in a column of $C$ all have the same sign.
\end{enumerate}
\end{proposition}
\begin{proof}
(1) follows since $A$ is Auslander regular by \cite[Theorem 4.6]{KKM} and (2) is a consequence of \cite[proposition 10.1]{Kle}.
\end{proof}

Now we made a sage program that computed all possible lower triangular 0-1 matrices $W$ with ones on the diagonal (corresponding to possible Cartan matrices of finite dimensional naturally labeled acyclic quiver algebras) satisfying the property that for $C:=-W^{T} W^{-1}$ we have $C=PU$ for a permutation matrix $P$ and an upper triangular matrix $U$ and furthermore that the non-zero entries in a column of $C$ all have the same sign. Next we used ChatGPT to produce, if possible, for each such $W$ a naturally labelled diagonal Auslander regular algebra $A$ with Cartan matrix $W$ and then we used QPA to check if this algebra is indeed diagonal Auslander regular and whether $A^{op}$ is also diagonal Auslander regular. We remark here that we used our main theorem \ref{thm:mainresult} to have a new quick test for the Auslander regular property in QPA, which we present in the appendix of this article. For $n=9$ vertices we found the first counterexample to Iyama's question that we present next.
Let $A=KQ/I$ with $K$ a field of characteristic 0 and $Q$ the quiver 
\[\begin{tikzcd}
	& 7 & 8 & 6 & \\
	1 & 3 && 5 & 9 \\
	& 2 && 4
	\arrow["j", from=1-2, to=1-3]
	\arrow["i"', from=1-4, to=1-3]
	\arrow["a", from=2-1, to=2-2]
	\arrow["e", from=2-2, to=1-2]
	\arrow["d"', from=2-2, to=2-4]
	\arrow["g"', from=2-4, to=1-4]
	\arrow["h"', from=2-4, to=2-5]
	\arrow["b", from=3-2, to=2-2]
	\arrow["c"', from=3-2, to=3-4]
	\arrow["f"', from=3-4, to=2-4]
\end{tikzcd}\]
and the relations $I=\langle adg,fh,ae,bd-cf,ej-dgi \rangle.$ $A$ is a naturally labelled quiver algebra with vector space dimension 31. 
\begin{proposition} \label{counterexampleiyamaquestion}
The algebra $A$ as above is diagonal Auslander regular, but $A^{op}$ is not diagonal Auslander regular.   
\end{proposition}
\begin{proof}
We first give the minimal projective resolutions of the indecomposable injective $A$-modules:
\begin{align*}
0&\longrightarrow P(8)
\longrightarrow P(6)\oplus P(7)
\longrightarrow P(3)
\longrightarrow P(1)
\longrightarrow I(1)\longrightarrow0,\\[2mm]
0&\longrightarrow P(5)
\longrightarrow P(3)\oplus P(4)
\longrightarrow P(2)
\longrightarrow I(2)\longrightarrow0,\\[2mm]
0&\longrightarrow P(6)
\longrightarrow P(3)\oplus P(4)
\longrightarrow P(1)\oplus P(2)
\longrightarrow I(3)\longrightarrow0,\\[2mm]
0&\longrightarrow P(9)
\longrightarrow P(3)
\longrightarrow P(2)
\longrightarrow I(4)\longrightarrow0,\\[2mm]
0&\longrightarrow P(3)
\longrightarrow P(1)\oplus P(2)
\longrightarrow I(5)\longrightarrow0,\\[2mm]
0&\longrightarrow P(7)
\longrightarrow P(2)
\longrightarrow I(6)\longrightarrow0,\\[2mm]
0&\longrightarrow P(4)
\longrightarrow P(2)
\longrightarrow I(7)\longrightarrow0,\\[2mm]
0&\longrightarrow P(2)
\longrightarrow I(8)\longrightarrow0,\\[2mm]
0&\longrightarrow P(1)
\longrightarrow I(9)\longrightarrow0.
\end{align*}
Next the following record the minimal injective coresolutions of the indecomposable projective $A$-modules:
\begin{align*}
0&\longrightarrow P(1)
\longrightarrow I(9)\longrightarrow0,\\[2mm]
0&\longrightarrow P(2)
\longrightarrow I(8)\longrightarrow0,\\[2mm]
0&\longrightarrow P(3)
\longrightarrow I(8)\oplus I(9)
\longrightarrow I(5)
\longrightarrow0,\\[2mm]
0&\longrightarrow P(4)
\longrightarrow I(8)
\longrightarrow I(7)
\longrightarrow0,\\[2mm]
0&\longrightarrow P(5)
\longrightarrow I(8)\oplus I(9)
\longrightarrow 
I(5)\oplus I(7)
\longrightarrow I(2)
\longrightarrow0,\\[2mm]
0&\longrightarrow P(6)
\longrightarrow I(8)
\longrightarrow I(5)\oplus I(7)
\longrightarrow I(3)
\longrightarrow0,\\[2mm]
0&\longrightarrow P(7)
\longrightarrow I(8)
\longrightarrow I(6)
\longrightarrow0,\\[2mm]
0&\longrightarrow P(8)
\longrightarrow I(8)
\longrightarrow I(6)\oplus I(7)
\longrightarrow I(3)
\longrightarrow I(1)
\longrightarrow0,\\[2mm]
0&\longrightarrow P(9)
\longrightarrow I(9)
\longrightarrow I(5)
\longrightarrow I(4)
\longrightarrow0.
\end{align*}

We leave the elementary verification that those are indeed minimal projective resolutions and minimal injective coresolutions to the reader.

This gives us all information that we need: We can read off the projective dimensions of the indecomposable injective $A$-modules, the grade and cogrades of the simple modules and their projective and injective dimensions using Lemma \ref{lem:minimal-ext-simple} and its dual to compute Ext spaces involving simples.
Note here that for algebras of finite global dimension, one has for a module $M$ that $\pdim M=\sup \{ i \geq 0 \mid \Ext_A^i(M,A) \neq 0 \}$ and dually $\idim M = \sup \{ i \geq 0 \mid \Ext_A^i(D(A),M) \neq 0 \}$, see for example Lemma 5.5 in \cite[Chapter VI.]{ARS}. 
We record all those homological dimension in the following table:

\begin{center}
\[
\renewcommand{\arraystretch}{1.18}
\begin{array}{c|ccccc}
i & \pdim_A S(i) & \idim_A S(i) & \gr_A S(i)
& \cogr_A S(i) & \pdim_A I(i)\\
\hline
1 & 3 & 0 & 3 & 0 & 3\\
2 & 2 & 0 & 2 & 0 & 2\\
3 & 2 & 1 & 2 & 1 & 2\\
4 & 2 & 1 & 2 & 1 & 2\\
5 & 1 & 2 & 1 & 2 & 1\\
6 & 1 & 2 & 1 & 2 & 1\\
7 & 1 & 2 & 1 & 1 & 1\\
8 & 0 & 3 & 0 & 3 & 0\\
9 & 0 & 2 & 0 & 2 & 0
\end{array}
\]
\end{center}

Now an algebra $A$ of finite global dimension is Auslander regular if and only if $\pdim I(i)=\gr S(i)$ for all $i$, see for example \cite[theorem 1.1]{KMT}. It is furthermore diagonal Auslander regular if additionally $\pdim S(i)= \gr S(i)$ for all $i$.
The table confirms that we indeed have $\pdim I(i)=\gr S(i)=\pdim S(i)$ for all $i$ and $A$ is diagonal Auslander regular.
Now with $A$ also $A^{op}$ is Auslander regular, but $A^{op}$ being diagonal Auslander regular is equivalent to the dual condition that all simple modules satisfy $\cogr S(i)=\idim S(i)$. We see that the simple module $S(7)$ does not satisfy this and thus $A^{op}$ is not diagonal Auslander regular.

\end{proof}

At the end, we remark that the following problem is still open:
\begin{question}
Let $A=KP$ be the incidence algebra of a finite poset. If $A$ is right diagonal Auslander regular, is it also left diagonal Auslander regular?
\end{question}
This question has a positive answer for finite lattices $L$ as there being diagonal Auslander regular for the incidence algebra is equivalent to $L$ being distributive and this condition is left-right symmetric, see \cite{IM}.
This question also has a positive answer for all posets with at most 12 elements as was verified by computer computations.

\section{Appendix: QPA code}

In this appendix we give code for the GAP package QPA~\cite{QPA}
implementing the criteria used in this paper. Throughout, $A=KQ/I$ is a
finite-dimensional quiver algebra and $n$ is a prescribed positive integer
serving as a homological bound. The function
\texttt{IsAuslanderGorenstein} assumes that every indecomposable injective
$A$-module has projective dimension at most $n$. The function
\texttt{GradeOfModule} assumes that $\pdim_A M\leq n$, while
\texttt{IsDiagonalAuslanderRegular} assumes that $\gldim A\leq n$.
If one of the required dimensions exceeds $n$, the corresponding function
returns \texttt{false}. Thus, a negative result is conclusive only when $n$
is known to be a valid bound. For an acyclic quiver algebra, one may take the
number of simple modules as a global dimension bound, see \cite{MS}.

The first function uses Theorem~\ref{thm:mainresult}. For every
indecomposable injective module $I$, it computes $d=\pdim_A I$ and identifies
the final nonzero term in the minimal projective resolution of $I$ with
$\Omega^d(I)$. Since $A$ is a basic quiver algebra, this projective module is
indecomposable precisely when the sum of the entries of the dimension vector
of its top is one. The position of the nonzero entry determines the image of
$I$ under the Auslander--Reiten map. The map is bijective precisely when these
positions are pairwise distinct.

The second function computes the grade from
\[
 \Ext_A^i(M,A)\cong
 \Ext_A^1(\Omega^{i-1}(M),A)\qquad (i\geq 1).
\]
The third function first verifies the prescribed
global-dimension bound, then tests the Auslander--Gorenstein property, and
finally checks that every simple module is perfect.

\newpage

\begin{scriptsize}
\begin{verbatim}
LoadPackage("qpa");;

IsAuslanderGorenstein := function(A,n)
    local vertices, I, d, lastTerm, topVector, vertex;

    if not IsPosInt(n) then
        Error("n must be a positive integer");
    fi;

    vertices := [];

    for I in IndecInjectiveModules(A) do
        d := ProjDimensionOfModule(I,n);
        if not IsInt(d) then
            return false;
        fi;

        # The d-th syzygy is the final projective term.
        lastTerm := NthSyzygy(I,d);
        topVector := DimensionVector(TopOfModule(lastTerm));

        # Its top is simple exactly when the projective is indecomposable.
        if Sum(topVector) <> 1 then
            return false;
        fi;

        vertex := Position(topVector,1);
        if vertex in vertices then
            return false;
        fi;
        Add(vertices,vertex);
    od;

    return true;
end;

GradeOfModule := function(A,M,n)
    local d, regularModule, i, E;

    if not IsPosInt(n) then
        Error("n must be a positive integer");
    fi;

    d := ProjDimensionOfModule(M,n);
    if not IsInt(d) then
        return false;
    fi;

    regularModule := DirectSumOfQPAModules(
        IndecProjectiveModules(A));

    if Length(HomOverAlgebra(M,regularModule)) > 0 then
        return 0;
    fi;

    for i in [1..d] do
        E := ExtOverAlgebra(
            NthSyzygy(M,i-1),regularModule);

        # A zero Ext^1-space may be returned in either of two forms.
        if Length(E) >= 2 and Length(E[2]) > 0 then
            return i;
        fi;
    od;

    return false;
end;

IsDiagonalAuslanderRegular := function(A,n)
    local S, d, g;

    if not IsPosInt(n) then
        Error("n must be a positive integer");
    fi;

    if not IsInt(GlobalDimensionOfAlgebra(A,n)) then
        return false;
    fi;

    if not IsAuslanderGorenstein(A,n) then
        return false;
    fi;

    for S in SimpleModules(A) do
        d := ProjDimensionOfModule(S,n);
        g := GradeOfModule(A,S,n);

        if not IsInt(g) or g <> d then
            return false;
        fi;
    od;

    return true;
end;
\end{verbatim}
\end{scriptsize}

We apply the code to the algebra from
Proposition~\ref{counterexampleiyamaquestion}, over $K=\mathbb{Q}$. 

\begin{scriptsize}
\begin{verbatim}
Q := Quiver(9,[
    [1,3,"a"], [2,3,"b"], [2,4,"c"], [3,5,"d"], [3,7,"e"],
    [4,5,"f"], [5,6,"g"], [5,9,"h"], [6,8,"i"], [7,8,"j"]
]);;
kQ := PathAlgebra(Rationals,Q);;
AssignGeneratorVariables(kQ);;
A := kQ/[a*e, f*h, a*d*g, b*d-c*f, d*g*i-e*j];;
n := Length(SimpleModules(A));;
Aop := OppositeAlgebra(A);;

Dimension(A);
IsDiagonalAuslanderRegular(A,n);
IsDiagonalAuslanderRegular(Aop,n);
\end{verbatim}
\end{scriptsize}

The three commands return, respectively, \texttt{31}, \texttt{true}, and
\texttt{false}. Thus $A$ is diagonal Auslander regular, whereas $A^{\op}$ is
not. The test for the Auslander--Gorenstein property uses our well-defined
Auslander--Reiten bijection result directly, thereby avoiding a much harder check of the Auslander condition by the direct definition; compare the earlier routines in the
appendix of \cite{GKKM}.

\section*{Acknowledgements}
We profited from the use of \cite{QPA}, \cite{Sage} and ChatGPT-5.6 Sol.

This project has received funding from the European Union's Horizon Europe research and innovation programme under the Marie Sk{\l}odowska-Curie grant agreement No 101126554. Vikt\'oria Kl\'asz and Markus Kleinau are supported by the Deutsche Forschungsgemeinschaft
(DFG, German Research Foundation) under Germany's Excellence Strategy - GZ 2047/1, Projekt-ID 390685813.

\section*{Disclaimer}
Co-Funded by the European Union. Views and opinions expressed are however those of the authors only and do not necessarily reflect those of the European Union. Neither the European Union nor the granting authority can be held responsible for them.

\noindent \includegraphics[height = 1cm]{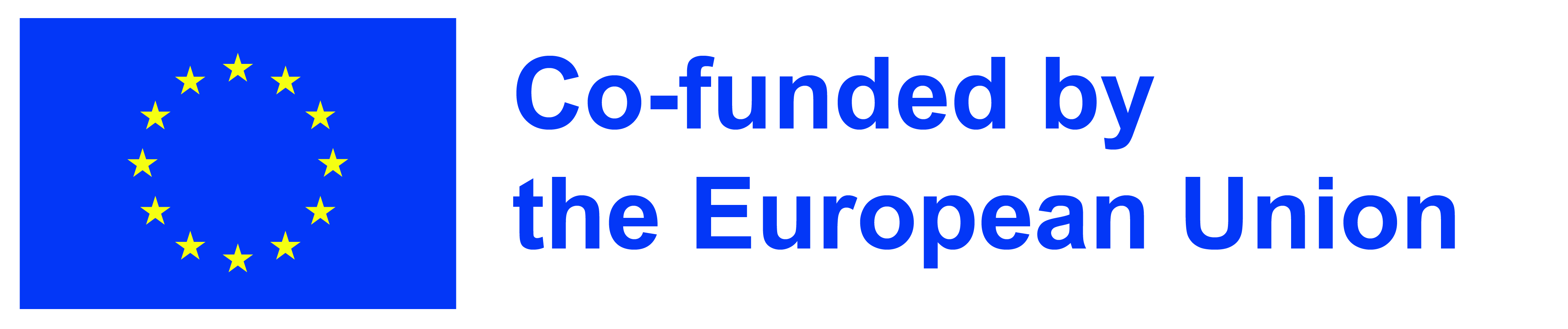}

\section*{Statement on the use of AI.}
Generative AI tools, namely ChatGPT-5.6 Sol, were used during the
exploratory and preparatory stages of this work. The mathematical strategy and proofs were conceived and directed by the authors.
ChatGPT assisted with selected aspects of the technical development, including
the formulation of auxiliary lemmas and the elaboration of technical details,
as well as with literature searches, consistency checks, and LaTex
preparation. All AI-generated suggestions were independently reviewed and
verified by the authors, who assume full responsibility for the mathematical
arguments, results, and final content of the paper.

\end{document}